\documentclass[11pt,fleqn]{amsart}
\usepackage[paper=a4paper]
 {geometry}

\usepackage[small]{titlesec}
\usepackage{hyperref}

\usepackage{amsthm,thmtools}
\usepackage[utf8]{inputenc}
\usepackage[english]{babel}
\usepackage{enumerate}
\usepackage[osf,noBBpl]{mathpazo}
\usepackage[alphabetic,initials]{amsrefs}
\usepackage{amsfonts,amssymb,amsmath}
\usepackage{mathtools}
\usepackage{graphicx}
\usepackage[poly,arrow,curve,matrix]{xy}
\usepackage{wrapfig}
\usepackage{xcolor}
\usepackage{helvet}
\usepackage{stmaryrd}
\usepackage{tikz}
\usepackage{mathdots}
\usepackage{arydshln}

\declaretheoremstyle[
bodyfont=\itshape,]{mystyle}
\declaretheorem[name=Lemma, style=mystyle, numberwithin=section]{Lemma}
\declaretheorem[name=Proposition, style=mystyle, sibling=Lemma]{Proposition}
\declaretheorem[name=Theorem, style=mystyle, sibling=Lemma]{Theorem}
\declaretheorem[name=Corollary, style=mystyle, sibling=Lemma]{Corollary}
\declaretheorem[name=Definition, style=mystyle, sibling=Lemma]{Definition}

\declaretheoremstyle[numbered=no, 
bodyfont=\itshape]{mystyle-empty}
\declaretheorem[name=Lemma, style=mystyle-empty]{Lemma*}
\declaretheorem[name=Proposition, style=mystyle-empty]{Proposition*}
\declaretheorem[name=Theorem, style=mystyle-empty]{Theorem*}
\declaretheorem[name=Corollary, style=mystyle-empty]{Corollary*}
\declaretheorem[name=Definition, style=mystyle-empty]{Definition*}
\declaretheorem[name=Example(s), style=mystyle-empty]{Example*}
\declaretheorem[name=Remark, style=mystyle-empty]{Remark*}

\renewcommand\thesection{\arabic{section}}
\titleformat{\section}
 {\normalfont\bfseries\large}
 {\thesection. \space}{0em}{}

\renewcommand\thesubsection{\arabic{subsection}}
\titleformat{\subsection}
 {\normalfont\bfseries}
 {\thesection.\thesubsection \space}{0em}{}

\renewcommand\proofname{Proof}

\makeatletter
\renewenvironment{proof}[1][\textit{\proofname}]{\par
 \pushQED{\qed}%
 \normalfont \topsep.75\paraskip\relax
 \trivlist
 \item[\hskip\labelsep
 \itshape
 #1\@addpunct{.}]\ignorespaces
}{%
 \popQED\endtrivlist\@endpefalse
}
\makeatother

\newskip\paraskip
\newcounter{para}[section]
\renewcommand\thepara{\thesection.\arabic{para}}
\def\paragraph{%
 \noindent
 \refstepcounter{para}%
 \textbf{\thepara.}\hspace{1ex}%
}

\newcommand\about[1]{%
 {\bfseries#1.}%
}

\newcommand\NN{\mathbb N}
\newcommand\CC{\mathbb C}
\newcommand\LL{\mathcal L}
\newcommand\RR{\mathcal R}

\newcommand\ot{\otimes}
\renewcommand\to{\longrightarrow}
\newcommand\g{\mathfrak g}
\newcommand\Id{\mathsf{Id}}

\newcommand\interval[1]{\underline{\mathbf #1}}
\newcommand\leftact{\triangleright}
\newcommand\rightact{\triangleleft}

\newcommand\re[1]{\overline{#1}}
\newcommand\retwo[1]{[\re{#1}]}
\newcommand\rethree[1]{\llbracket\re{#1}\rrbracket}

\DeclareMathOperator\End{End}
\DeclareMathOperator\Ret{Ret}
\DeclareMathOperator\mpl{mpl}

\title{A note on set theoretical solutions of the Yang-Baxter equation with 
trivial retraction}
\author{Pablo Zadunaisky}
\address{Instituto de Investigaciones Matemáticas Luis Santaló (IMAS) -- 
  CONICET
  Ciudad Universitaria – Pabellón I – Piso 2 – Oficina 2040
  (C1428EGA) – C.A.B.A., Argentina } 
\email{pzadunaisky@conicet.gov.ar}
\thanks{The author is a CONICET fellow. This work was supported in part by the 
projects UBACYT 20020170100613BA and PICT-2019-00099}

\begin{document}

\begin{abstract}
We show that every finite non-degenerate set theoretical solution to the YBE 
whose retraction is a flip linearizes to a twist of the flip solution by roots 
of unity. This generalizes a result of Gateva-Ivanova and Majid. To prove the 
result we use a new invariant associated to a solution, its Lie algebra. We 
show also that a solution retracts to a flip solution if and only if its Lie 
algebra is abelian.

\end{abstract}

\maketitle

\section{Introduction}
A set theoretical solution to the Yang-Baxter equation is a pair $(X,r)$ where
$X$ is a set and $r: X \times X \to X \times X$ is a map satisfying
\begin{align*}
r \times \Id_X \circ \Id_X \times r \circ r \times \Id_X
  &=  \Id_X \times r \circ r \times \Id_X \circ  \Id_X \times r.
\end{align*}
Their study was suggested by Drinfeld in \cite{Drinfeld90}, and begun in 
earnest with the seminal works of Gateva-Ivanova and van den Bergh 
\cite{GIvdB98} and Etingof, Schedler and Soloviev \cite{ESS99}. To study these 
solutions one usually constructs associated algebraic objects, such as 
monoids, groups, quadratic algebras, Hopf algebras (all introduced in the 
cited papers) and more recently homology groups \cite{LV17}, braces 
\cite{Rump07}, and skew-braces \cite{GV17},among many others. For an overview 
of the literature and open problems in the area, and its connections to other 
areas of mathematics, we refer the reader to \cite{Vendramin19}.

The space of all set theoretical solutions of the YBE is vast and one usually 
focuses on special classes to avoid drowning in it. A typical result will 
prove that the nice combinatorial properties of the solutions are reflected in 
nice properties of the associated algebraic objects. Common families to study 
are involutive solutions, or solutions of small size, etc. This paper is 
concerned with a family of solutions which is close to trivial in the sense 
that its retraction is a flip solution (see section \ref{s:generalities} for 
definitions). These solutions have attracted much interest recently, mostly in the involutive case (i.e. $r^2 = \Id_{X \times X}$), see for example 
\cites{Rump22,JPZ20,JPZ21,JP23}. 

We present two results on solutions whose retraction is a flip solution, but 
are not assumed to be involutive. 
First, we show that the multipermutation level is reflected in the structure 
of the linearized solution, which must be a flip solution twisted by roots of 
unity. This result was already observed by Gateva-Ivanova and Majid for 
involutive square-free solutions in \cite{GIM11}*{Theorem 4.9}. Second, we 
introduce a new algebraic invariant, the Lie algebra associated to a solution, 
and show that a solution retracts to a flip solution if and only if its 
associated Lie algebra is abelian.

\subsubsection*{Convention.}
All vector spaces and tensor products are over $\CC$. Throughout this paper $X$
will denote a finite set.

\section{Generalities on solutions}
\label{s:generalities}
\paragraph
\about{Quadratic sets}
Following \cites{GI04a,GIM08}, a \emph{quadratic set} is a pair $(X,r)$ where 
$X$ is a set and $r: X \times X \to X \times X$ is a bijective map; its 
cardinality is the cardinality of $X$. For a fixed quadratic set $(X,r)$ and 
$x,y \in X$ we denote by $\LL_x, \RR_y$ the functions defined implicitly by
\[r(x,y) 
  = (\LL_x(y), \RR_y(x)).\]
It is sometimes useful to see $\LL_x$ and $\RR_y$ as defining actions of $X$
on itself, so we also use the alternative notation
\[
  r(x,y) = (x \leftact y, x \rightact y)  
\]
The quadratic set is \emph{non-degenerate} if the maps $\LL_x$ and $\RR_y$ are 
bijective for each $x \in X$, it is \emph{involutive} if $r^2 = \Id_X$, it is 
\emph{square-free} if $r(x,x) = (x,x)$ for all $x \in X$, and finally it is 
\emph{finite} if the set $X$ is finite. As per our convention, in this paper 
we will only consider finite quadratic sets.

Set $r_{1} = r \times \Id_X$ and $r_{2} = \Id_X \times r$, which are 
functions of $X \times X \times X$ to itself. The quadratic set $(X,r)$ is a 
solution to the set theoretical Yang-Baxter equation, or to be brief just 
\emph{a solution}, if $r_{1} \circ r_{2} \circ r_{1} = r_{2} \circ r_{1} \circ 
r_{2}$.

\paragraph
\about{Examples}
\label{examples}
Suppose $X$ is any finite set. If we choose arbitrary functions $f,g:X \to X$
then $r: X \times X \to X \times X$ given by $r(x,y) = (f(y),g(x))$ is a 
quadratic set. It is a solution if and only if $f$ and $g$ commute, it is non 
degenerate if and only if both $f$ and $g$ are bijective, and it is involutive 
if and only if $g = f^{-1}$. These solutions are known as \emph{permutation 
solutions} and were suggested by Lyubashenko \cite{Drinfeld90}. The 
\emph{flip} solution on $X$ is the permutation solution with $f = g = \Id_X$,
that is the function $r(x,y) = (y,x)$. 

We will use the following two examples in the sequel. For $n \in \NN$ we 
denote by $\interval n$ the set $\{1,2,\ldots, n\}$. Take $s: \interval 2 
\times \interval 2 \to \interval 2 \times \interval 2$ given by 
\begin{align*}
\begin{tabular}{c|cc}
$s$ & $1$ & $2$ \\ 
\hline
$1$ & $(2,2)$ & $(1,2)$ \\
$2$ & $(2,1)$ & $(1,1)$ \\
\end{tabular}
\end{align*}
This solution is involutive and non-degenerate but not square-free. It is a
permutation solution with $f = g = (12)$.

Now take $t: \interval 8 \times \interval 8 \to \interval 8 \times \interval 8$
to be the solution with
\begin{align*}
\begin{tabular}{c|c:c:c:c}
$x$ & $1$ \  $2$ & $3$ \  $4$ & $5$ \  $6$ & $7$ \  $8$\\
\hline
$\LL_x$ & $e$ & $(34)(78)$ & $(34)(78)$ & $e$ \\
$\RR_x$ & $e$ & $(34)(56)$ & $e$ & $(34)(56)$
\end{tabular}
\end{align*}
A long but direct computation shows that this is a solution, which is clearly 
not a multipermutation solution. Since $t^2(3,5) = t(5,3) = (4,6)$, $t$ is not 
an involutive solution, and since $t(3,3) = (4,4)$ it is not square-free.

\begin{Remark*}
In the literature the expression \emph{trivial} solution is reserved for flip 
solutions with $|X| \geq 2$, while the only solution with $|X| = 1$ is a 
one-element solution. To avoid clashing with the established notation we will 
use the expression ``flip solutions'' to encompass both trivial and 
one-element solutions.
\end{Remark*}

\paragraph
\about{Retractions}
We define an equivalence relation in $X$ denoted by $\sim$, where $x \sim y$ 
if and only if $\LL_x = \LL_y$ and $\RR_x = \RR_y$. We denote by $\re x$ the 
equivalence class of $x$ and by a slight abuse of notation we write 
$\LL_{\re x} = \LL_x$ and $\RR_{\re x} = \RR_x$. The following lemma is 
standard. It was proved for non-degenerate involutive solutions by Etingof, 
Schedler and Soloviev \cite{ESS99}*{Subsection 3.2}. We denote by $\re X$ 
the quotient $X /\!\sim$. 

\begin{Lemma}
Suppose $(X,r)$ is a quadratic set. Then $(X,r)$ is a solution if and only if 
the following equalities hold in $\End(V)$
\begin{align*}
\LL_{\re x} \circ \LL_{\re y}
  &= \LL_{\re{x \leftact y}} \circ 
    \LL_{\re{x \rightact y}}; \\
\RR_{\re z} \circ \RR_{\re y} 
  &= \RR_{\re{y \rightact z}} \circ 
    \RR_{\re{y \leftact z}}; \\
\RR_{\re{(x \rightact y) \leftact z}} \circ \LL_{\re x} \circ 
    p_{\re y}
  &= \LL_{\re{x \rightact (y \leftact z)}} \circ \RR_{\re z} \circ 
    p_{\re y}.
\end{align*}
If $r$ is non-degenerate, it induces a map $\re r: \re X \times \re X \to 
\re X \times \re X$ given by $\re r(\re x, \re y) = (\re{x \leftact y}, 
\re{x \rightact y})$ and the quadratic set $(\re X, \re r)$ is also a 
solution.
\end{Lemma}
\begin{proof}
For each $x,y,z \in X$ we have
\begin{align*}
r_1 r_2 r_1(x,y,z)
  &= (\LL_{\re{x \leftact y}}(\LL_{\re{x \rightact y}}(z)),
      \RR_{\re{(x \rightact y) \leftact z}}(\LL_{\re x}(y)),
      \RR_{\re z}(\RR_{\re y}(x))); \\
r_2 r_1 r_2(x,y,z)
  &= (\LL_{\re x} (\LL_{\re y}(z)), 
      \LL_{\re{x \rightact (y \leftact z)}}(\RR_{\re z}(y)),
      \RR_{\re{y \rightact z}}(\RR_{\re{y \leftact z}}(x))).
\end{align*}
Since $y = p_{\overline y}(y)$ the equalities in the statement are equivalent
to the fact that $r$ is a solution to the YBE.

Suppose now that $r$ is non-degenerate if $\re x = \re x'$ then $\re{x 
\leftact y} = \re{x' \leftact y}$. Thus
\[
\LL_{\re{x \rightact y}} 
= \LL_{\re x} \circ \LL_{\re y} \circ (\LL_{\re{x \leftact y}})^{-1}
= \LL_{\re{x'}} \circ \LL_{\re y} \circ (\LL_{\re{x' \leftact y}})^{-1} =
\LL_{\re{x' \rightact y}}
\]
and similarly $\RR_{\re{x \rightact y}} = \RR_{\re{x' \rightact y}}$. Thus 
$\re{x \rightact y}$ depends only on $\re y$ and $\re x$. Similarly
$\re{y \leftact x} = \re{y \leftact x'}$ and this result only depends on 
$\re y$ and $\re x$. This implies that the map $\re r: \re X \times \re X \to
\re X \times \re X$ given by $\re r(\re x, \re y) = (\re{x \leftact y}, 
\re{x \rightact y})$ is well-defined. Since $r$ is a solution, so is 
$\overline r$.
\end{proof}
The solution $\Ret(X,r) = (\re X, \re r)$ is called the \emph{retraction} of
$(X,r)$. We say that a finite solution $(X,r)$ is retractable if $|\re X| < 
|X|$. It may happen that the retraction of a retractable solution is itself 
retractable, so we set inductively $\Ret^{n+1}(X,r) = \Ret^n(\re X, \re r)$. 
The \emph{multipermutation level} of $(X,r)$, denoted $\mpl(X,r)$, is the 
infimum of all $n \in \NN$ such that $\Ret^n(X,r)$ is a solution of 
cardinality $1$. Notice that a solution is a permutation solution if and only 
if it is of multipermutation level $1$.

\begin{Example*}
Any permutation solution retracts to a one-element solution (which under our 
convention is a flip solution). Example $t$ from \ref{examples} has $\re 1 = 
\re 2, \re 3 = \re 4, \re 5 = \re 6$ and $\re 7 = \re 8$. The solution 
$\Ret(X,r)$ is the flip solution over the set of classes. 
\end{Example*}

\begin{Remark*}
There is ambiguity in the notation $\LL_{\re x}$, as it could be a function 
from $X$ to itself or from $\re X$ to itself. This issue will not arise in the 
sequel as we will always work at the level of $X$ and its linearization $V$.
\end{Remark*}

\paragraph
\about{Associated vector space}
We denote by $V$ the $\CC$-span of $X$. The space $V$ is a Hermitian space 
with the unique inner product making $X$ an orthogonal basis. Given a set $S 
\subset V$ we denote by $p_S$ the orthogonal projection to the space generated 
by $S$. If $(X,r)$ is a quadratic set the maps $\LL_{\re x}$ and $\RR_{\re x}$ 
induce linear endomorphisms of $V$, which we denote by the same symbols. 
Finally, we define the map $R: V \otimes V \to V \otimes V$ as the obvious 
linear extension of $r$, namely $R(x \otimes y) = u \otimes v$ whenever 
$r(x,y) = (u,v)$ for $x,y,u,v \in X$.  

\paragraph
\about{Formulas}
\label{formulas}
From now on we assume that $(X,r)$ is a non-degenerate solution. To 
avoid overly complicated notation we denote by $\retwo{x}$ the class of $\re 
x$ in  $\Ret^2(X, r)$, and by $\rethree x$ the class of $\retwo x$ in 
$\Ret^3(X, r)$. Since $x \leftact y$ depends only on the class of $x$ it makes 
sense to write $\re x \leftact y$, and by a similar reasoning $\re{\re x 
\leftact y} = \retwo{x} \leftact \re y$, etc. We can thus rewrite the 
conditions on the previous lemma as
\begin{align*}
\LL_{\re x} \circ \LL_{\re y}
  &= \LL_{\retwo x \leftact \re y} \circ 
    \LL_{\re x \rightact \retwo y}; \\
\RR_{\re z} \circ \RR_{\re y} 
  &= \RR_{\re y \rightact \retwo z} \circ 
    \RR_{\retwo y \leftact \re z}; \\
\RR_{(\retwo x \rightact \rethree y) 
  \leftact \re z} \circ \LL_{\re x} \circ 
    p_{\re y}
  &= \LL_{\re x \rightact (\rethree y \leftact 
  \retwo z)} \circ \RR_{\re z} \circ 
    p_{\re y}.
\end{align*}
We also point out that
\begin{align*}
\LL_{\re x} \circ p_{\re y} 
  &= p_{\retwo x \leftact \re y} \circ \LL_{\re x};
&\RR_{\re z} \circ p_{\re y} 
  &= p_{\re y \rightact \retwo z} \circ \RR_{\re z};
&p_{\re x} \circ p_{\re y} = p_{\re y} \circ p_{\re x}
= \delta_{\re x, \re y} p_{\re x}.
\end{align*}
Finally, we can write $R$ as
\[
  R \circ \tau 
    = \sum_{\re x \in \re X} \sum_{\re y \in \re X}
    \LL_{\re x} \circ p_{\re y} \ot \RR_{\re y} \circ 
    p_{\re x} 
\]
where $\tau$ is the flip map given by $\tau(v \ot w) = w \ot v$. To see
this, notice that given $v,w \in X$ we have
\begin{align*}
\bigg(\sum_{\re x \in \re X} \sum_{\re y \in \re X}
    &\LL_{\re x} \circ p_{\re y} \ot \RR_{\re y} \circ 
    p_{\re x} \bigg)\tau(v \ot w)\\
    &= \sum_{\re x \in \re X} \sum_{\re y \in \re X}
    \LL_{\re x} \circ p_{\re y}(w) \ot \RR_{\re y} \circ 
    p_{\re x}(v) \\
    &= \LL_{\re v}(w) \ot \RR_{\re w}(v)
    = R(v \ot w),
\end{align*}
so the desired equality holds over a basis of $V \ot V$.

\section{Multipermutation level $2$}
Recall that here and below $(X,r)$ denotes a finite non-degenerate solution.

\paragraph\about{Formulas for multipermutation level $2$}
If $\mpl(X,r) \leq 2$ then $\retwo y = X$. Also there exist $f,g: \re X \to 
\re X$ such that $\re r(\re x, \re y) = (f(\re y), g(\re x))$ so the formulas 
above simplify to
\begin{align*}
\LL_{\re x} \circ \LL_{\re y}
  &= \LL_{f(\re y)} \circ \LL_{g(\re x)}; 
& \RR_{\re z} \circ \RR_{\re y} 
  &= \RR_{g(\re y)} \circ \RR_{f(\re z)}; 
&\RR_{f(\re z)} \circ \LL_{\re x} 
  &= \LL_{g(\re x)} \circ \RR_{\re z};\\
\LL_{\re x} \circ p_{\re y} 
  &= p_{f(\re y)} \circ \LL_{\re x}
&\RR_{\re x} \circ p_{\re y} 
  &= p_{g(\re y)} \circ \RR_{\re x}
\end{align*}
In particular, if $\Ret(X,r)$ is a flip solution, or equivalently if $f = 
g = \Id_{\re X}$, then the maps $\LL_{\re x}, \RR_{\re y}, p_{\re z}$ commute 
for all $x,y,z \in X$.

\paragraph
The following result is a generalization of \cite{GIM11}*{Theorem 4.9} to 
non-involutive solutions. Also, we replace the hypothesis that $r$ is 
square-free and of multipermutation level $2$ with the weaker hypothesis that 
the first retraction is a flip solution. 
\begin{Proposition}
\label{P:q-polynomials}
Let $(X,r)$ be a finite non-degenerate solution, and suppose $\Ret(X,r)$ is 
a flip solution. Then there exist a basis $\{v_1, \ldots, v_n\}$ of $V$ and a 
set of roots of unity $\mathbf q = (q_{i,j})_{1 \leq i,j \leq n} \subset \CC$ 
such that $R(v_j \otimes v_i) = q_{i,j} v_i \otimes v_j$. 
\end{Proposition}
\begin{proof}
Since $p_{\re z}$ is a projector it is diagonalizable, and its 
eigenvalues are either $1$ or $0$. Since $\LL_{\re x}$ and 
$\RR_{\re y}$ are induced by bijections of the set $X$, they are also 
diagonalizable and their eigenvalues are roots of unity. As mentioned above 
the fact that $\Ret(X,r)$ is a flip solution means that the maps $\LL_{\re x}, 
\RR_{\re y}$ commute, and they commute with $p_{\re z}$, for all 
$\re x, \re y, \re z \in \re X$. Also since $X$ is the 
disjoint union of the classes $\re z$ the projectors $p_{\re z}$ 
are orthogonal and add up to the identity, so they also commute with each 
other.

By a standard result of linear algebra the fact that these linear operators 
are diagonalizable and commute implies that they are \emph{simultaneously} 
diagonalizable, i.e. there exists a basis of $V$ consisting of eigenvectors of 
all these maps, say $\{v_1, \ldots, v_n\}$. Write $\lambda_{i,\re x}$, resp. 
$\mu_{i,\re y}$, for the eigenvalue of $v_i$ with respect to $\LL_{\re x}$, 
resp. $\RR_{\re y}$. Using the fact that $R = \left( \sum_{\re x, \re y \in 
\re X} \LL_{\re x} \circ p_{\re y} \ot \RR_{\re y} \circ p_{\re x} \right) 
\circ \tau$ and that for each $i$ there is exactly one $\re{z(i)}$ such that 
$p_{\re{z(i)}}(v_i) \neq 0$ we get
\begin{align*}
R(v_j \ot v_i) = \lambda_{i,\re{z(j)}} \mu_{j,\re{z(i)}} v_i \ot v_j
\end{align*}
Since $\lambda_{i,\re{z(j)}}$ and $\mu_{j,\re{z(i)}}$ are roots of unity, so 
is their product and we are done.
\end{proof}

Notice that this proof is the only place where we use that the base field is 
$\CC$, to guarantee that the minimal polynomial of the linear transformation 
induced by a permutation splits. 

\paragraph
The algebra $Q(X,r)$ associated to a solution is the quotient of the tensor 
algebra $T(V)$ by the ideal generated by the image of the map $\Id_{V \otimes 
V} - R$. Following \cite{Manin88} this algebra is called the Yang-Baxter 
algebra associated to the solution. As shown in \cite{GIvdB98}*{Corollary 
1.5}, this algebra is integral whenever $(X,r)$ is involutive. In fact by 
\cite{JKVA19}*{Theorem 4.5} if $(X,r)$ is finite, bijective and left 
non-degenerate then $Q(X,r)$ is integral if and only if $(X,r)$ is involutive.

In the special case where $(X,r)$ is a finite non-degenerate solution whose 
retraction is a flip solution then by Proposition \ref{P:q-polynomials} its
Yang-Baxter algebra $Q(X,r)$ is isomorphic to the $q$-polinomial algebra 
$\CC_{\mathbf q}[v_1, \ldots, v_n]$, where $\mathbf q = 
(q_{i,j})_{1 \leq i,j \leq n}$ are the roots of unity from the statement, and 
the relations $v_j v_i = q_{i,j} v_i v_j$ hold for all $i,j$. This allows us
to check that this is a domain if and only if $(X,r)$ is involutive very 
easily. 

\begin{Corollary}
\label{c:q-pols}
Let $(X,r)$ be a finite non-degenerate solution, and suppose 
$\Ret(X,r)$ is a flip solution. Then $Q(X,r)$ is finite over its center. 
Furthermore, $Q(X,r)$ is a domain if and only if $(X,r)$ is involutive.
\end{Corollary}
\begin{proof}
As we have already observed, $Q(X,r)$ is the quotient of $T(V)$ by the 
relations $v_i v_j = q_{j,i} v_j v_i$. Now if $d_i$ is the minimal natural 
number such that $q_{i,j}^{d_i} = 1$ for all $j$, then $y_i = v_i^{d_i}$ is a 
central element in $Q(X,r)$. Thus $\Gamma = \CC[y_1, \ldots, y_n] \subset 
Q(X,r)$ is a central subalgebra and $Q(X,r)$ is a finite module over $\Gamma$.

The algebra $\CC_{\mathbf q}[v_1, \ldots, v_n]$ is integral if and only if
$q_{i,i} = 1$ and $q_{i,j} = q_{j,i}^{-1}$ for all $1 \leq i,j, \leq n$ (see 
for example \cite{MR01}*{Chapter 1 \S 6}). Now
\begin{align*}
R^2(v_i \otimes v_j) &= q_{i,j}q_{j,i} v_i \otimes v_j,
\end{align*}
so $(X,r)$ is involutive if and only if $q_{i,j} = q_{j,i}^{-1}$. Thus if
$Q(X,r)$ is integral then $(X,r)$ is involutive, and if $(X,r)$ is involutive
we only need to show that $q_{i,i} = 1$. 

Let $v_i$ be any of the vectors in the eigenbasis from the proposition. Then
$v_i$ lies in the span of $\re x$ for one and only one $\re x \in \re X$. The 
hypothesis that $\Ret(X,r)$ is a flip solution implies that $(\re x, 
r|_{\re x \times \re x})$ is a solution, and indeed a permutation solution 
since $r|_{\re x \times \re x} = \LL_{\re x}|_{\re x} \times \RR_{\re 
x}|_{\re x}$. Furthermore, if $(X,r)$ is involutive then so is $(\re x, 
r|_{\re x \times \re x})$, and this implies $\RR_{\re x}|_{\re x} = 
(\LL_{\re x}|_{\re x})^{-1}$. This equality also holds for the linear 
extensions of these maps, and so in the notation of the proposition
\begin{align*}
R(v_i \otimes v_i)
  &= \LL_{\re x}(v_i) \otimes \RR_{\re x}(v_i)
  = \LL_{\re x}(v_i) \otimes (\LL_{\re x})^{-1}(v_i)
  = \lambda_{i,\re x} v_i \otimes \lambda_{i, \re x}^{-1} v_i
  = v_i \otimes v_i,
\end{align*}
so $q_{i,i} = 1$.
\end{proof}

The Yang-Baxter algebra has been studied intensively, notably by 
Gateva-Ivanova, see for example the articles \cites{GIvdB98,GI04,GI04a,GI12,
GI18,GI23} and the references therein. In particular \cite{GI18} gives 
conditions on the algebraic structures of an involutive solution $(X, r)$ (not 
necessarily finite) which determine the multipermutation level of $(X, r)$, 
and the last sections contains several results on multipermutation solutions 
of level $2$. 

\begin{Remark*}
Similar results hold for the Hopf algebra associated to the solution as 
defined in \cite{ESS99}*{section 2.9}). These results appear in the MsC thesis 
of Agust\'\i n Mu\~noz \cite{Munoz}. 
\end{Remark*}

\paragraph
\about{Examples}
Recall the involutive, non square-free solution $(\interval 2, s)$ from 
\ref{examples}. Taking $v_1 = x_1 + x_2$ and $v_2 = v_1 - v_2$ the 
corresponding linear solution is given by 
\begin{align*}
\begin{tabular}{c|cc}
$R$ & $v_1$ & $v_2$ \\ 
\hline
$v_1$ & $v_1 \otimes v_1$ & $-v_2 \otimes v_1$ \\
$v_2$ & $-v_1 \otimes v_2$ & $v_2 \otimes v_2$ \\
\end{tabular}
\end{align*}
The corresponding Yang-Baxter algebra is then the polynomial algebra in 
anticommuting variables.

Now consider solution $(\interval 8, t)$, which has multipermutation level 
$2$ and is neither involutive nor square free. By Corollary \ref{c:q-pols} 
$Q(X,r)$ is a $q$-polynomial ring but not a domain. Using the notation from 
the proposition we have eigenvectors 
\begin{align*}
v_1 &= x_1; 
  &v_2 &= x_2; 
  &v_3 &= x_3 + x_4;
  &v_4 &= x_3 - x_4;\\
v_5 &= x_5 + x_6;
  &v_6 &= x_5 - x_6;
  &v_7 &= x_7 + x_8;
  &v_8 &= x_7 - x_8.
\end{align*}
All $\lambda_{i, \re x}$ and $\mu_{i, \re x}$ are $1$ except
\begin{align*}
\lambda_{4,\re 3} &= \lambda_{4, \re 5} = \lambda_{8,\re 3} 
  = \lambda_{8, \re 5} = -1\\
\mu_{4,\re 3} &= \mu_{4, \re 7} = \mu_{6,\re 3} = \mu_{6, \re 7} =-1 
\end{align*}
Thus for example $q_{3,8} = \lambda_{3, \re 8} \mu_{8, \re 3} = 1$ but 
$q_{8,3} = \lambda_{8, \re 3} \mu_{3, \re 8} = -1$. This implies $v_8v_3 = 
v_3v_8 = -v_8v_3 = -v_3v_8$ from which follows that $v_8v_3 = v_3v_8 = = 0$.

\section{A Lie algebra associated to a solution}
Throughout this section $(X,r)$ denotes a finite non-degenerate solution. We 
associate to $(X,r)$ a Lie algebra contained in $\mathfrak{gl}(V)$ and show 
that it characterizes solutions whose retracion is a flip solution.

\begin{Definition}
The Lie algebra associated to the solution $(X,r)$, denoted by $\g(X,r)$, is 
the Lie subalgebra of $\mathfrak{gl}(V)$ generated by the operators 
$\{\LL_{\re x} \circ p_{\re y}, \RR_{\re x} \circ p_{\re y} \mid \re x, \re y 
\in \re X\}$.
\end{Definition}

\paragraph
From the formulas in \ref{formulas} we deduce the commutation formulas
\begin{align*}
[\LL_{\re x} \circ p_{\re y}, \LL_{\re z} \circ p_{\re w}]
  &= p_{\retwo x \leftact \re y} 
    \circ p_{\rethree x \leftact (\retwo z \leftact \re w)}
    \circ \LL_{\re x} \circ \LL_{\re z} - p_{\retwo z \leftact \re w} 
    \circ p_{\rethree z \leftact (\retwo x \leftact \re y)}
    \circ \LL_{\re z} \circ \LL_{\re x},
\\
[\RR_{\re x} \circ p_{\re y}, \RR_{\re z} \circ p_{\re w}]
  &= p_{\re y \rightact \retwo x} 
    \circ p_{(\re w \rightact \retwo z) \rightact \rethree x}
    \circ \RR_{\re x} \circ \RR_{\re z} - p_{\re w \rightact \retwo z} 
    \circ p_{(\re y \rightact \retwo x) \rightact \rethree z}
    \circ \RR_{\re z} \circ \RR_{\re x},
\\
[\LL_{\re x} \circ p_{\re y}, \RR_{\re z} \circ p_{\re w}]
  &= p_{\retwo x \leftact y} 
    \circ p_{\rethree x \leftact(\re w \rightact \retwo z)} \circ 
      \LL_{\re x} \circ \RR_{\re z} -
    p_{\re w \rightact \retwo z} 
    \circ p_{(\retwo x \leftact \re y) \rightact \rethree z} \circ 
      \RR_{\re z} \circ \LL_{\re x}. 
\end{align*}
It als follows that the linear solution $R$
is a solution to the classical YBE in $\g(X,r)$. The crucial fact in the proof 
of Proposition \ref{P:q-polynomials} is that the operators $\{\LL_{\re x} 
\circ p_{\re y}, \RR_{\re z} \circ p_{\re w}\}$ commute when $\Ret(X,r)$ is a 
flip solution, so in this case $R$ is a solution to the YBE in an abelian Lie 
algebra. 

\begin{Theorem}
Let $(X,r)$ be a finite non-degenerate solution and let $\g = \g(X,r)$. Then 
the following are equivalent.
\begin{enumerate}[(a)]
\item \label{i:trivial-ret} The solution $\Ret(X,r)$ is a flip solution.
\item \label{i:q-poly} The span of the set $\{\LL_{\re x}, \RR_{\re x}, 
p_{\re x}\}$ is an abelian Lie subalgebra of $\mathfrak{gl}(V)$.
\item \label{i:g-abelian} The algebra $\g(X,r)$ is abelian.
\end{enumerate}
\end{Theorem}
\begin{proof}
The implication $(\ref{i:trivial-ret}) \Rightarrow (\ref{i:q-poly})$ was shown 
in the proof of Proposition \ref{P:q-polynomials}, and is also easy to deduce 
from the commutation formulas displayed above. The implication $(\ref{i:q-poly})
\Rightarrow (\ref{i:g-abelian})$ is immediate. 

Now suppose $\g(X,r)$ is abelian. Choose arbitrary elements $x,y,z \in X$ 
and choose $w \in X$ such that $y = w \rightact z$. Since $[\LL_{\re x} \circ 
p_{\re y}, \RR_{\re z} \circ p_{\re w}] = 0$ we deduce from the formulas
in \ref{formulas} that
\begin{align*}
p_{\retwo x \leftact y} \circ \LL_{\re x} \circ \RR_{\re z}
  &= p_{\re y} \circ p_{(\retwo x \leftact \re y) \rightact \rethree z} 
    \circ \RR_{\re z} \circ \LL_{\re x}.
\end{align*}
Since $\LL_{\re x}$ and $\RR_{\re z}$ are automorphisms of $V$, the morphism
on the left hand side of the equality is nonzero and its image is the subspace 
of $V$ spanned by $\retwo x \leftact \re y$. Thus the map on the right hand 
side is also nonzero, which implies that $\re y = (\retwo x \leftact \re y) 
\rightact \rethree z$, and its image is the subspace spanned by $\re y$, which
implies that $\re y = \retwo x \leftact \re y$. These 
imply that $\re y$ is a fixed point for the action of $\retwo x$ on the left
and for the action of $\retwo z$ on the right. Since $x,y,z$ were chosen 
arbitrarily, this means that $\LL_{\retwo x} = \RR_{\retwo z} = \Id_{\re V}$
and hence $\Ret(X,r)$ is a flip solution.
\end{proof}

We would like to have a more intrinsic definition of $\g(X,r)$ but so far
this is the best description we have. Computer experiments show that for all 
$|X| \leq 8$ and $r$ involutive the derived algebra of $\g(X,r)$ is semisimple 
of type $A$. Thus the multipermutation level is not directly related to the 
solubility or nilpotency of $\g(X,r)$, as one might suppose from the 
statement. We intend to return to this matter in future work.

\subsubsection*{Acknowledgements}
I thank Leandro Vendramin and Tatiana Gateva-Ivanova for reading and 
commenting on an early version of this article. I would also like to thank the 
referee for their comments. The computer experiments mentioned above were ran 
by Agust\'\i n Mu\~noz as part of his MsC. thesis, using Vendramin's 
\textsf{GAP} package YBE.

\end{document}